\documentclass{article}
\usepackage[a4paper, left=3cm, right=3cm, top=3cm]{geometry}
\usepackage{amssymb}
\usepackage{amsmath}
\usepackage{amsthm}
\usepackage{tikz}
\usepackage{verbatim}
\usepackage{overpic}
\usepackage{bbm} 
\usepackage{setspace}
\usepackage{float} 

\title{Noise dependent synchronization of a degenerate SDE}
\author{Isabell Vorkastner \thanks{
	Institut f\" ur Mathematik, MA 7-5, Fakult\" at II, Technische Universit\" at Berlin, 
	Stra\ss e des 17. Juni 136, 10623 Berlin, Germany; \texttt{vorkastn@math.tu-berlin.de} }}
\date{July 2016}

\begin{document}

\theoremstyle{definition}

\newtheorem{theorem}{Theorem}[section]
\newtheorem{lemma}[theorem]{Lemma}
\newtheorem{proposition}[theorem]{Proposition}
\newtheorem{corollary}[theorem]{Corollary}
\newtheorem{definition}[theorem]{Definition}
\newtheorem{remark}[theorem]{Remark}
\newtheorem{example}[theorem]{Example}

\setcounter{page}{1}

\maketitle

\begin{abstract}
	\noindent
	We provide an example of an SDE with degenerate additive noise where synchronization depends on 
	the strength of noise and the number of directions in which the noise acts. 
	Here, synchronization means that the weak random attractor consists of a single random point.
	Indicated by a change of sign of the top Lyapunov exponent, we prove synchronization
	respectively no (weak) synchronization. 
	
\end{abstract}

\section{Introduction}

We consider the stochastic differential equation with drift given by a multidimensional 
double-well potential with degenerate additive noise. That is
\begin{equation} \label{sde}
\begin{split}
	dX_t &= \left( X_t - \left| \left( \begin{array}{c} X_t \\ Y_t \end{array} \right)	 \right| ^{2}
		X_t \right) dt + \sigma \, dW_t  \quad \textrm{on} \; \mathbb{R}^n \\
	dY_t &= \left( Y_t \; - \left| \left( \begin{array}{c} X_t \\ Y_t \end{array} \right)	 \right| ^{2} 
		Y_t \; \right) dt \qquad \qquad \quad \textrm{on} \; \mathbb{R}^{d-n}
\end{split}
\end{equation}
for $\sigma > 0$ and $d,n \in \mathbb{N}$ with $ n < d$. 
$W_t$ is a $n$-dimensional Brownian motion.
We prove that the associated random dynamical system does synchronize in the case $n=1$ for large $\sigma$ and in the case $n \geq 2$. Additionally, we show that there is no synchronization,
not even weak synchronization, in the case $n=1$ for small $\sigma$.
Here, (weak) synchronization means that there exists a weak (point) attractor which is a single random point. 
Thus, the long-time dynamics are asymptotically globally stable.
In particular, in case of (weak) synchronization, for each $x,y \in \mathbb{R}^d$ and $\omega \in \Omega$ 
it follows that
\begin{align*}
	\left| \varphi_t (\omega, x) - \varphi_t (\omega, y) \right| \rightarrow 0 \qquad \textrm{as } t \rightarrow \infty
\end{align*}
in probability. 
\\
In the deterministic case, for $\sigma =0$, the long-time dynamics are not asymptotically globally stable. The attractor in this case is the closed unit ball $\bar{B}(0,1)$.
Moreover, the minimal point attractor is given by $S^{d-1} \cup \left\{ 0 \right\}$, where $S^{d-1}$ is 
the $( d-1 )$-dimensional unit sphere.
Hence, there will be no (weak) synchronization. 
\\
In \cite{SynNoise} it was shown that under some conditions a general white noise random dynamical system 
on a complete, separable metric space does synchronize. 
There, the stochastic differential equation with drift given by a multidimensional double-well potential 
with non-degenerate additive noise was considered as a model example. 
Hence, the proofs in \cite{SynNoise} imply synchronization 
of the random dynamical system in the case of non-degenerate noise.
\\
For the stochastic differential equation with degenerate noise (\ref{sde}) we observe 
no synchronization in the case $n=1$ for small noise and
synchronization in the other cases. These results confirm that additive noise stabilizes
the long-time dynamics of the multidimensional double-well potential. 
The distinction between synchronization and no (weak) synchronization is indicated by a change of
sign of the top Lyapunov exponent.
\\
In order to prove synchronization in the degenerate case with negative top Lyapunov exponent, 
we follow the setup put forward in \cite{SynNoise}. 
In \cite{SynNoise} asympotic stability, 
swift transitivity and contraction on large sets were used to prove synchronization of 
a white noise random dynamical system.
However, the random dynamical system associated to (\ref{sde}) is not swift transitive. 
This can be seen by observing that the set
$\left\{ (x_1, x_2, \dots , x_d) \in \mathbb{R}^d : x_i > 0 \right\}$ is not reachable if one starts in 
$\left\{ (x_1, x_2, \dots , x_d) \in \mathbb{R}^d : x_i < 0 \right\}$ for some $n < i \leq d$.
We will deal with the lack of swift transitivity by focusing on elements of the set
$M = \left\{ (x_1,x_2, \dots, x_d) \in \mathbb{R}^d : x_i=0 \textrm{ for } i>n \right\}$. 
Asymptotic stability is obtained by a stable manifold theorem and negative top Lyapunov exponent.
\\
A positive top Lyapunov exponent of the random dynamical system associated to (\ref{sde}) implies
lack of (weak) synchronization. In general, attractors with positive top Lyapunov exponent are not well
understood yet. These attractors are sometimes called random strange attractors 
\cite{Bifurcation, Ledrappier1988, PhysRevE}. 
\\
Recently, in \cite{Bifurcation} a transition from negative to positive top Lyapunov exponent was shown too. 
They considered a system with limit cycles on a cylinder perturbed by white noise. Using \cite{SynNoise}
synchronization was also proven. However, they did not state whether there is synchronization
for positive top Lyapunov exponent.
\\
The remainder of the paper is organized as follows.
In Section 2, we introduce some notation and definitions. 
We conclude existence of a random dynamical system and of an attractor in Section 3. 
In Section 4, we estimate the top Lyapunov exponent. Using negative top Lyapunov exponent and
a stable manifold theorem, we conclude asymptotic stability. 
The actual proof of synchronization appears in Section 5. For this purpose, contraction
on large sets and some similar property to swift transitivity are shown. 
These two properties and asymptotic stability are used to prove synchronization.
In Section 6, we show that the positive top Lyapunov exponent imply no (weak) synchronization.

\section{Notation and definition}

We restrict our definitions to a random dynamical system on $\mathbb{R}^d$, 
see \cite{arnold} for a more general setting.
\begin{definition} [Metric Dynamical System]
	Let $(\Omega, \mathcal{F}, \mathbb{P})$ be a probability space 
	and $\theta = (\theta_t )_{t \in \mathbb{R}}$ be a group of 
	maps $\theta_t : \Omega \rightarrow \Omega$ satisfying
	\begin{enumerate}
	\item [(i)] $(\omega,t) \mapsto \theta_t (\omega)$ is 
		$(\mathcal{F} \otimes \mathcal{B}(\mathbb{R}), \mathcal{F})$-measurable,
	\item [(ii)] $\theta_0 (\omega) = \omega$ for all  $\omega \in \Omega$,
	\item [(iii)] $\theta_{s+t} = \theta_s \circ \theta_t$ for all $s,t \in \mathbb{R}$,
	\item [(iv)] $\theta_t$ has ergodic invariant measure $\mathbb{P}$.
	\end{enumerate}
	Then, $(\Omega, \mathcal{F}, \mathbb{P}, \theta)$ is called a \textit{metric dynamical system}.
\end{definition}

\begin{definition} [Random Dynamical System]
	Let $(\Omega, \mathcal{F}, \mathbb{P}, \theta)$ be a metric dynamical system. 
	Further, let $\varphi : \mathbb{R}_+ \times \Omega \times \mathbb{R}^d \rightarrow \mathbb{R}^d$ 
	be such that
	\begin{enumerate}
	\item [(i)] $\varphi$ is $(\mathcal{B}(\mathbb{R}_+) \otimes \mathcal{F} \otimes 
		\mathcal{B}(\mathbb{R}^d), \mathcal{B}(\mathbb{R}^d))$-measurable,
	\item [(ii)] $\varphi_0 (\omega, x) = x$ for all $x \in \mathbb{R}^d$, $\omega \in \Omega$,
	\item [(iii)] $\varphi_{t+s} (\omega,x)= \varphi_t ( \theta_s \omega , \varphi_s (\omega,x))$
		for all $x \in \mathbb{R}^d$, $t,s \geq 0$, $\omega \in \Omega$,
	\item [(iv)] $x \mapsto \varphi_s(\omega, x)$ is continuous for each $s \geq 0$ and $\omega \in \Omega$.
	\end{enumerate}
	Then, the collection $(\Omega, \mathcal{F}, \mathbb{P}, \theta, \varphi )$ 
	is called a \textit{random dynamical system} (RDS).
\end{definition}

\noindent
As an example consider an RDS generated by a stochastic differential equation (SDE) 
driven by a Brownian motion. 
In order to use the white noise property of the Brownian motion, 
the existence of a family $\mathbb{F}= (\mathcal{F}_{s,t})_{- \infty < s \leq t < \infty}$ 
of sub-$\sigma$ algebras  of $\mathcal{F}$ will be desirable. 
This family of sub-$\sigma$ algebras should satisfy $\mathcal{F}_{t,u} \subset \mathcal{F}_{s,v}$ 
for $s \leq t \leq u \leq v$, $\theta_r^{-1} (\mathcal{F}_{s,t}) = \mathcal{F}_{s+r,t+r}$ 
for all $r,s,t$ and $\mathcal{F}_{s,t}$ and $\mathcal{F}_{u,v}$ are independent for
$s \leq t \leq u \leq v $. For each $t \in \mathbb{R}$ denote by $\mathcal{F}_t$ 
the smallest $\sigma$-algebra containing all $\mathcal{F}_{s,t}$ with $s \leq t$ 
and by $\mathcal{F}_{t, \infty}$ the smallest $\sigma$-algebra containig all 
$\mathcal{F}_{t,u}$ with $t \leq u$.
Note that the $\sigma$-algebras $\mathcal{F}_t$ and $\mathcal{F}_{t,\infty}$ 
are independent for all $t \in \mathbb{R}$.
Furthermore, assume that $\varphi_s (\cdot, x)$ is $\mathcal{F}_{0,s}$-measurable 
for each $s \geq 0$ and $x \in \mathbb{R}^d$.
Then, the collection $( \Omega, \mathcal{F}, \mathbb{F}, \mathbb{P}, \theta, \varphi )$ 
is called a \textit{white noise random dynamical system}.
For a white noise RDS $\varphi$ define the associated Markovian semigroup by
$P_t f (x) := \mathbb{E}\left[ f \left( \varphi_t \left( \cdot,x \right) \right) \right]$
for measurable, bounded fuctions $f$.

\begin{definition}
	A family $\left\{ D (\omega) \right\}_{\omega \in \Omega}$ of non-empty subsets 
	of $\mathbb{R}^d$ is said to be
	\begin{enumerate}
		\item [(i)]  a \textit{random compact set} if it is $\mathbb{P}$-almost surely compact 
		and $\omega \mapsto \sup_{y \in D(\omega)} |x - y|$ 
		is $\mathcal{F}$-measurable for each $x \in \mathbb{R}^d$.
		\item [(ii)] \textit{$\varphi$-invariant} if for all $t \geq 0$
		\begin{align*}
			\varphi_t (\omega, D(\omega)) = D(\theta_t \omega)
		\end{align*}
		for almost all $\omega \in \Omega$.
	\end{enumerate}	 
\end{definition}

\begin{definition}[Attractor]
	Let $(\Omega, \mathcal{F}, \mathbb{P}, \theta, \varphi)$ be an RDS. 
	A random compact set $A$ is called a \textit{pullback attractor} if 
	it satisfies the following properties
	\begin{enumerate}
	\item [(i)] A is $\varphi$-invariant
	\item [(ii)] for every compact set $B \subset \mathbb{R}^d$ 
		\begin{align*}
			\lim_{t \rightarrow \infty} \sup_{x \in B} \inf_{a \in A(\omega)}
				\left| \varphi_t (\theta_{-t}\omega ,x) - a \right| =0 
				\qquad \mathbb{P} \textrm{-almost surely}.
		\end{align*}
	\end{enumerate}
	A random compact set $ A$ is called a \textit{weak attraktor}, if it fulfills 
	the properties above with almost sure convergence replaced 
	by convergence in probability in (ii).
	It is called a \textit{(weak) point attractor}, if it satisfies the properties above 
	with compact sets $B$ replaced by single points in (ii).
\end{definition}

\noindent
Note that every pullback attractor is a weak attractor. The converse is not true. 
Examples for this can be found in \cite{ref1}.



\begin{lemma}
	\label{F_0}
	Let $A$ be a weak attractor of an RDS $\varphi$. 
	Then $A$ admits an $\mathcal{F}_0$-measurable version. 
	Hence, there exists an $\mathcal{F}_0$-measurable
	weak attractor $\tilde{A}$ such that $A = \tilde{A}$ $\mathbb{P}$-almost surely.
\end{lemma}

\begin{proof}
	Since $\varphi$ has a weak attractor, \cite[Corollary 4.5.]{Crauel2009} implies that 
	$\varphi$ has an $\mathcal{F}_0$-measurable weak attractor $\tilde{A}$.
	By \cite[Lemma 1.3]{SynNoise} $A = \tilde{A}$ $\mathbb{P}$-almost surely.  
\end{proof}

\begin{definition}[Synchronization]
	\textit{Synchronization} occurs if there is a weak attractor $A(\omega)$ being 
	a singleton for $\mathbb{P}$-almost every $\omega \in \Omega$.
	\textit{Weak synchronization} is said to occur if there is a weak point attractor $A(\omega)$ being 
	a singleton for $\mathbb{P}$-almost every $\omega \in \Omega$.
\end{definition}

\noindent
Define
\begin{align*}
	B(x,r):= \left\{ y \in \mathbb{R}^d : |y-x|<r \right\}
\end{align*}
to be the open ball centered at $x \in \mathbb{R}^d$ with radius $r>0$ 
and let $\bar{B}(x,r)$ be the respective closed ball.
Denote by 
\begin{align*}
	\textrm{diam} (A) := \sup_{x,y \in A} |x-y|.
\end{align*}
the diameter of a set $A \subset \mathbb{R}^d$.
Next, some properties of an RDS are defined. These properties were used 
in \cite{SynNoise} to show synchronization.
Note that asymptotic stability and contraction on large sets are necessary conditions.

\begin{definition}[Asymptotic Stability]
	Let $U \subset \mathbb{R}^d$ be a deterministic non-empty open set. 
	Then, $\varphi$ is called \textit{asymptotically stable  on $U$} if there exists 
	a deterministic sequence $t_n \rightarrow \infty$ such that
	\begin{align*}
		\mathbb{P} \left( \lim_{n \rightarrow \infty} 
		\textrm{diam} \left( \varphi_{t_n} \left( \cdot, U \right) \right) = 0 \right) > 0.
	\end{align*}
\end{definition}

\begin{definition}[Swift Transitivity]
	$\varphi$ is called \textit{swift transitive} if for every $r>0$ and 
	$x,y \in \mathbb{R}^d$, there is a time $t>0$ such that
		\begin{align*}
			\mathbb{P}(\varphi_{t} ( \cdot, B(x,r)) \subset B(y,2r)) > 0.
		\end{align*}
\end{definition}

\begin{definition}[Contraction on Large Sets]
	$\varphi$ is called \textit{contracting on large sets} if for every $r>0$, 
	there is a ball $B(x,r)$ and a time $t>0$ such that
		\begin{align*}
			\mathbb{P}\left( \textrm{diam}\left( \varphi_{t} 
				( \cdot, B(x,r))\right) \leq \frac{r}{4}\right) > 0.
		\end{align*}
\end{definition}	
\section{Existence of an attractor}

In this section we deduce existence of an RDS associated to (\ref{sde}) 
and of an attractor of this RDS.
Denote by $ b : \mathbb{R}^d \rightarrow \mathbb{R}^d$ with $ b(x) := (1 - |x|^2)x$ the drift of (\ref{sde}).

\begin{lemma} [One-sided Lipschitz condition] \label{oslc}
	The drift $b$ fulfills
	\begin{align*}
		\left\langle x - y, b(x) -b(y) \right\rangle \leq |x-y|^2 \left(1- \frac{3}{4} |x|^2\right)
	\end{align*}
	for all $x,y \in \mathbb{R}^d$. In particular, $b$ satisfies the one-sided Lipschitz condition.
\end{lemma}

\begin{proof}
	Let $x,y \in \mathbb{R}^d$ and define $a := x-y$. Using 
	$ |x-a|^2 = |x|^2 -2 \left\langle a,x \right\rangle + |a|^2$
	and the Cauchy - Schwarz inequality, it follows that
	\begin{align*}
		\left\langle x-y, b(x) - b(y) \right\rangle
			&= \left\langle x-y, x(1-|x|^2)- y(1-|y|^2) \right\rangle \\
			&= \left\langle a, a -|x|^2 x + |x-a|^2 (x-a) \right\rangle \\
			&= |a|^2 -(|x|^2 -2 \left\langle a,x \right\rangle + |a|^2) |a|^2 
				+ (-2 \left\langle a,x \right\rangle + |a|^2) \left\langle a, x \right\rangle \\
			&= |a|^2 - |a|^2 |x|^2 - \left( |a|^2 - \frac{3}{2} \left\langle a, x \right\rangle \right)^2
				+ \frac{1}{4} \, | \left\langle a, x \right\rangle |^2 \\
			&\leq |a|^2 - \frac{3}{4} |a|^2 |x|^2 .
	\end{align*}
\end{proof}

\begin{remark}
	\label{ex_rds}
	The drift $b$ of the SDE (\ref{sde}) satisfies the one-sided Lipschitz condition 
	by Lemma \ref{oslc} and a local Lipschitz condition 
	since $b$ is continuously differentiable. By \cite[Proposition 2.4]{Attractor}
	it follows that there exists a white noise RDS $\varphi$ associated to the SDE (\ref{sde})
	with respect to the canonical setup. This means the space 
	$\Omega$ is $C( \mathbb{R}, \mathbb{R}^n)$, the space of continuous functions 
	from $\mathbb{R}$ to $\mathbb{R}^n$, $\mathcal{F}$ is the Borel $\sigma$-field, 
	$\mathbb{P}$ is the two-sided Wiener measure, 
	$\mathcal{F}_{s,t}$ is the $\sigma$-algebra generated by $W_u - W_v$ for 
	$s \leq v \leq u \leq t$ with $W_s: \Omega \rightarrow \mathbb{R}^n$
	defined by $W_s (\omega) := \omega (s)$, and $\theta_t$ is the shift
	\begin{align*}
		(\theta_t \omega) (s) := \omega (s+t) - \omega (t).
	\end{align*}
\end{remark}

\begin{theorem} [Existence of an Attractor]
	\label{ex_attractor}
	There exists a pullback attractor $A$ of the RDS $\varphi$ associated to (\ref{sde}).
\end{theorem}

\begin{proof}
	By Lemma \ref{oslc} the drift $b$ satisfies the one-sided Lipschitz condition. 
	Since the drift $b$ is continuously differentiable, 
	$b$ is local Lipschitz continuous. Moreover,
	\begin{align*}
		\limsup_{|x| \rightarrow \infty} \left\langle \frac{x}{|x|} , b(x) \right\rangle 
		= \limsup_{|x| \rightarrow \infty} |x| (1- |x|^2) = - \infty.
	\end{align*}
	By \cite[Theorem 3.1.]{Attractor} it follows that $\varphi$ has a pullback attractor.
\end{proof}

\noindent
In the following denote by $\varphi$ the RDS associated to (\ref{sde}) and by $A$ 
the $\mathcal{F}_0$-measurable version of the weak attractor 
given by Theorem \ref{ex_attractor}. The existence of an $\mathcal{F}_0$-measurable version 
was shown in Lemma \ref{F_0}.	
\section{Top Lyapunov exponent and asymptotic stability}

We estimate the top Lyapunov exponent of the RDS associated to (\ref{sde}) 
and observe a change of sign. Applying a stable manifold theorem
and using negativity of the top Lyapunov exponent, asymptotic stability for the RDS associated to (\ref{sde}) is shown in the case $n=1$ for large $\sigma$ and in the case $n \geq 2$.

\begin{theorem}[Stable Manifold Theorem] 
	\label{smt}
	Let $\varphi_t(\omega,\cdot) \in C_{loc}^{1,\delta}$ for some $\delta \in (0,1)$ and 
	all $t\geq0$ and let $P_t$ be the Markovian semigroup associated to $\varphi$. 
	Assume that $P_1$ has an ergodic invariant measure $\rho$ such that
	\begin{align*}
		\mathbb{E} \int_{\mathbb{R}^d} \log^{+} \left\| D\varphi_1(\omega,x) \right\| \, d\rho(x) < \infty
	\end{align*}
	and \begin{align*}
		\mathbb{E} \int_{\mathbb{R}^d} \log^{+} \left\| \varphi_1(\omega,\cdot + x)
			- \varphi_1(\omega,x)\right\|_{C^{1,\delta}(\bar{B}(0,1))} \, d\rho(x) < \infty.
	\end{align*}
	Then,
	\begin{enumerate}
		\item[(i)]
			there are constants $\lambda_N < \ldots < \lambda_1$ such that 
		 	\begin{align*}
		 		\lim_{m \rightarrow \infty} \frac{1}{m} 
		 			\log |D\varphi_m (\omega,x)v| \in \left\{ \lambda_i \right\}_{i=1}^N
			\end{align*}
			for all $v \in \mathbb{R}^d \setminus \left\{0\right\}$ and 
			$\mathbb{P} \otimes \rho$-almost all $( \omega,x) \in \Omega \times \mathbb{R}^d$.
		 \item[(ii)]
		 	Assume that the top Lyapunov exponent $\lambda_{top} := \lambda_1 < 0$. Then, for every
		 	$\varepsilon \in (\lambda_{top},0)$
		 	there is a measurable map 
		 	$\beta: \Omega \times \mathbb{R}^d \rightarrow \mathbb{R}_+ \setminus \left\{ 0 \right\}$ 
		 	such that for $\rho$-almost all $x \in \mathbb{R}^d$
		 	\begin{align*}
		 		\mathcal{S} (\omega,x) := \left\{ y \in \mathbb{R}^d : | \varphi_m(\omega,y) 
		 			- \varphi_m(\omega,x)| \leq 
		 		\beta (\omega,x) e^{\varepsilon m} \textrm{ for all } m \in \mathbb{N} \right\}
			\end{align*}		 		
			is an open neighborhood of $x$ $\mathbb{P}$-almost surely.
	\end{enumerate}
\end{theorem}
\begin{proof}
	See Lemma 3.1. in \cite{SynNoise}.
\end{proof}

	\noindent
	From the stable manifold theorem (Theorem \ref{smt}) one obtains a random, non-empty, open set 
	$\mathcal{S} (\omega,x)$. One aims to show asymptotic stability on a deterministic, 
	non-empty, open set. The following lemma clarifies the relation between the random set 
	$\mathcal{S} (\omega,x)$ and the existence of a deterministic set $U$ such that $\varphi$ 
	is asymptotically stable on $U$.

\begin{lemma}
	\label{det_set}
	Let $V$ be a random open neighborhood of $x \in \mathbb{R}^d$ and let $t_n \rightarrow \infty$ 
	be a sequence such that
	\begin{align*}
		\mathbb{P} \left( \lim_{m \rightarrow \infty } \textrm{diam} \left( 
		\varphi_{t_m} (\cdot, V(\cdot)) \right) =0 \right) > 0. 
	\end{align*}
	Then, there exists some deterministic $r>0$ such that
	\begin{align*}
		\mathbb{P} \left( \lim_{m \rightarrow \infty } 
			\textrm{diam} \left( \varphi_{t_m} (\cdot, B(x,r)) \right) =0 \right) > 0. 
	\end{align*}
	In particular, $\varphi$ is asymptotically stable on $B(x,r)$.
\end{lemma}
\begin{proof}
	For each $\omega \in \Omega$ there exists $k \in \mathbb{N}$ such that 
	$B \left( x,\frac{1}{k} \right) \subset V(\omega)$. Hence,
	\begin{align*}
	\begin{split}
		&\left\{ \lim_{m \rightarrow \infty} \textrm{diam} 
		\left( \varphi_{t_m} (\cdot,V(\cdot))\right) =0 \right\} \\
		& \qquad \subset \left\{ \lim_{m \rightarrow \infty} \textrm{diam} \left( \varphi_{t_m} 
			\left( \cdot,B \left( x, \frac{1}{k} \right) \right) \right) =0 
			\textrm{ for some } k \in \mathbb{N} \right\}.
	\end{split}
	\end{align*}
	By $\sigma$-additivity of $\mathbb{P}$ there exists some $r>0$ such that
	\begin{align*}
		\mathbb{P} \left( \lim_{m \rightarrow \infty } \textrm{diam} 
		\left( \varphi_{t_m} (\cdot, B(x,r)) \right) =0 \right) > 0. 
	\end{align*}
\end{proof}

\begin{remark}
	\label{invariant_measure}
	To apply the stable manifold theorem (Theorem \ref{smt}), finiteness of the above stated expectations
	and an invariant measure is required. 
	For an invariant measure consider the $n$-dimensional double-well potential 
	with non-degenerate additive noise. That is
	\begin{equation} \label{sdeng}
	\begin{split}
		d X_t &= \left( X_t - \left| X_t	 \right| ^{2} X_t \right) dt 
			+ \sigma \, dW_t  \quad \textrm{on} \; \mathbb{R}^n .
	\end{split}
	\end{equation}
	By \cite[Theorem, p. 243]{InvMea} and since 
	the function $f(x) = e^{\frac{2}{\sigma^2}(\frac{1}{2} |x|^2 - \frac{1}{4} |x|^4) }$ satisfies 
	the equation $0 = \left( \mathcal{L}^\ast f \right) (x) = 
	\frac{1}{2} \sigma^2 \triangle f(x) - \nabla \cdot \left( f(x)( 1 - |x|^2) x \right)$, 
	the Markovian semigroup associated to the RDS of (\ref{sdeng}) has the invariant probability measure 
	\begin{align*}
		d \hat{\rho}(x) = \frac{1}{Z_\sigma}  
		e^{\frac{2}{\sigma^2}(\frac{1}{2} |x|^2 - \frac{1}{4} |x|^4) } \, dx ,
	\end{align*}
	where $Z_\sigma = \int_{\mathbb{R}^n} e^{\frac{2}{\sigma^2}(\frac{1}{2} |x|^2 
	- \frac{1}{4} |x|^4) } \, dx$.
	Considering our original $d$-dimensional SDE with degenerate noise (\ref{sde}), if one starts in 
	$M := \left\{ (x_1, x_2, \dots ,x_d) \in \mathbb{R}^d : x_i = 0 \textrm{ for } i>n \right\}$ 
	one will stay in $M$. Hence, starting in $M$ the problem 
	simplifies to the $n$-dimensional non-degenerate case (\ref{sdeng}). 
	Therefore, the measure $\rho$ on $\mathbb{R}^d$ with
	\begin{align*}
		\rho \left(  A \times \left\{ 0 \right\} \right)
			= \hat{\rho} (A)
			=  \frac{1}{Z_\sigma} \int_A  e^{\frac{2}{\sigma^2}
				(\frac{1}{2} |x|^2 - \frac{1}{4} |x|^4) } \, dx 
	\end{align*}
	for all $A \in \mathcal{B} \left( \mathbb{R}^{n} \right)$ and
	\begin{align*}
		&\rho \left( \mathbb{R}^{n} \times  
		\left( \mathbb{R}^{d-n} \setminus\left\{ 0 \right\} \right) \right) =0 
	\end{align*}
	is an invariant probability measure of the Markovian semigroup associated to the RDS of (\ref{sde}).
\end{remark}

\begin{lemma}
	\label{assumptions}
	The RDS $\varphi$ associated to (\ref{sde}) satisfies $\varphi_t (\omega, \cdot) \in C_{loc}^2$,
	\begin{align*}
		\mathbb{E} \int_{\mathbb{R}^d} \log^{+} \left\| D\varphi_1(\omega,x) \right\| \, d\rho(x) < \infty
	\end{align*}
	and 
	\begin{align*}
		\mathbb{E} \int_{\mathbb{R}^d} \log^{+} \left\| \varphi_1(\omega,\cdot + x)
			- \varphi_1(\omega,x)\right\|_{C^{1,\delta}(\bar{B}(0,1))}
			\, d\rho(x) < \infty.
	\end{align*}
\end{lemma}
\begin{proof}
	Let $t>0$, $\omega \in \Omega$ and $x,u \in \mathbb{R}^d$.
	The derivatives of the drift $b$ satisfy
	\begin{align}
	\label{est_b1}
	\begin{split}
		\left\langle Db(x) u, u \right\rangle 
		&= -2 \left| \left\langle x, u \right\rangle \right|^2 + \left( 1 - |x|^2 \right) | u |^2 \\
		&\leq \left( 1 - |x|^2 \right) |u|^2 
		\leq |u|^2
	\end{split}
	\end{align}
	and
	\begin{align*}
		\left\| D^2 b(x) \right\| \leq 6 |x|.
	\end{align*}
	Moreover, 
	\begin{align*}
		\int_{\mathbb{R}^d} \log^+ (|x|) \, d \rho (x)
		= \int_{\mathbb{R}^d} \log^+ (|x|) 
			e^{\frac{2}{\sigma^2}(\frac{1}{2} |x|^2 - \frac{1}{4} |x|^4) } \, dx < \infty
	\end{align*}
	by rapidly decaying property of $e^{\frac{2}{\sigma^2}(\frac{1}{2} |x|^2 - \frac{1}{4} |x|^4) }$.
	Then, the second estimate and $\varphi_t (\omega, \cdot) \in C_{loc}^2$ follow by the same
	arguments as in \cite[Lemma 3.9]{SynNoise}. 
	In \cite[Lemma 3.9]{SynNoise} SDEs with non-degenerate additive noise were considered. 
	However, the arguments extend to SDEs with degenerate additive noise.
	To get the first estimate observe that
	\begin{align*}
		\frac{d}{dt} D \varphi_t (\omega, x) = Db (\varphi_t (\omega,x)) 
		D\varphi_t (\omega,x), \quad D\varphi_0 (\omega,x) = \textrm{Id}.
	\end{align*}
	Using (\ref{est_b1}) it follows that
	\begin{align*}
		\frac{d}{dt} \left| D \varphi_t (\omega, x) v \right|^2 
		&= 2 \left\langle Db (\varphi_t (\omega,x)) D\varphi_t (\omega,x) v, 
			D \varphi_t (\omega, x) v \right\rangle\\
		&\leq 2 \left| D \varphi_t (\omega, x) v \right|^2.
	\end{align*}
	By Gronwall's inequality
	\begin{align*}
		\left| D \varphi_t (\omega, x) v \right| \leq \left|v \right| e^{t}.
	\end{align*}
	Hence, $\left\| D \varphi_t (\omega, x) \right\| \leq  e^{t}$ and
	\begin{align*}
		\mathbb{E} \int_{\mathbb{R}^d} \log^{+} \left\| D\varphi_1(\omega,x) \right\| \, d\rho(x) 
		\leq \int_{\mathbb{R}^d} \log^{+} \left( e^1 \right)  \, d\rho(x) =1 < \infty.
	\end{align*}
\end{proof}

\begin{lemma}
	\label{le1}
	The top Lyapunov exponent of the RDS associated to (\ref{sde}) corresponding to the 
	invariant measure $\rho$ (see Remark \ref{invariant_measure}) satisfies
	\begin{align*}
		\lambda_{top} \leq \frac{1}{Z_\sigma} \int_{\mathbb{R}^n} (1 - |x|^2) \, 
			e^{\frac{2}{\sigma^2}(\frac{1}{2} |x|^2 - \frac{1}{4} |x|^4) } \, dx ,
	\end{align*}
	where $Z_\sigma = \int_{\mathbb{R}^n} 
	e^{\frac{2}{\sigma^2}(\frac{1}{2} |x|^2 - \frac{1}{4} |x|^4) } \, dx$.
	For $n = 1$ even equality holds.	
\end{lemma}
\begin{proof}
	\textit{Step 1:}
	In the first step, it will be shown that for some $\omega \in \Omega$ and \linebreak
	$x \in M := \left\{ (x_1, x_2, \dots ,x_d) \in \mathbb{R}^d : x_i = 0 \textrm{ for } i>n \right\}$ 
	it holds that
	\begin{align*}
		\lambda_{top} &\leq \liminf_{m \rightarrow \infty} \frac{1}{m} \int_0^m 
		(1 - |\varphi_s(\omega,x)|^2) \, ds.
	\end{align*}
	By Theorem \ref{smt} (i) there exists an $v \in \mathbb{R}^d \setminus \left\{ 0 \right\}$, $x \in M$ 
	and $\omega \in \Omega$ such that
	\begin{align*}
		\lambda_{top} = \lim_{m \rightarrow \infty} \frac{1}{m} \log | D \varphi_m(\omega,x)v|.
	\end{align*}
	$D\varphi_t (\omega,x)$ satisfies the equation
	\begin{align*}
		\frac{d}{dt} D \varphi_t (\omega,x) = 
		Db ( \varphi_t(\omega,x)) D\varphi_t (\omega,x), \quad D\varphi_0 (\omega,x)= Id.
	\end{align*}
	Using the estimation (\ref{est_b1}) it follows that
	\begin{align*}
		\frac{d}{dt} |D \varphi_t (\omega,x)v|^2 
		&= 2 \left\langle Db ( \varphi_t(\omega,x)) D\varphi_t (\omega,x)v, 
			D\varphi_t (\omega,x)v \right\rangle \\
		&\leq 2 \left( 1 - \left| \varphi_t(\omega,x)\right|^2 \right) | D\varphi_t (\omega,x)v|^2.
	\end{align*}
	By Gronwall's inequality,
	\begin{align*}
		| D\varphi_t (\omega,x)v| 
		\leq |v| \, e^{ \int_0^t \left( 1 - \left| \varphi_s(\omega,x)\right|^2 \right) \, ds}.
	\end{align*}
	Hence,
	\begin{align*}
		\lambda_{top} \leq \liminf_{m \rightarrow \infty} \frac{1}{m} 
		\int_0^m \left( 1 - \left| \varphi_s(\omega,x)\right|^2 \right) \, ds.
	\end{align*}
	\textit{Step 2:}
	Let $x \in M$ and $\omega \in \Omega$.
	For $n=1$ it will be shown that
	\begin{align*}
		\lambda_{top} 
		&\geq \lim_{m \rightarrow \infty} \frac{1}{m} \int_0^m (1 - |\varphi_s(\omega,x)|^2) \, ds.
	\end{align*}
	In the case $n=1$, $Db(y)= (1 - |y|^2) Id -2 y \otimes y$ is a diagonal matrix for all $y \in M$.
	Moreover, $ \varphi_t (\omega,x) \in M$ and
	\begin{align*}
		\frac{d}{dt} D \varphi_t (\omega,x)v 
		= Db ( \varphi_t(\omega,x)) D\varphi_t (\omega,x)v, \quad D\varphi_0 (\omega,x)v= v
	\end{align*}
	for any $v \in \mathbb{R}^d \setminus \left\{ 0 \right\}$. Denote by $(\cdot)^{(i)}$ the i-th component
	of a vector. Then, for any $v \in \mathbb{R}^d \setminus \left\{ 0 \right\}$
	\begin{align*}
		(D \varphi_t (\omega,x)v)^{(1)} = v^{(1)} \, e^{ \int_0^t (1 - 3 |\varphi_s(\omega,x)|^2) \, ds}
	\end{align*}
	and
	\begin{align*}
		(D \varphi_t (\omega,x)v)^{(i)} = v^{(i)} \, e^{ \int_0^t (1 - |\varphi_s(\omega,x)|^2) \, ds}
	\end{align*}
	for $i>1$. Choose $v = (0, \ldots, 0,1)^T \in \mathbb{R}^d$. Then, 
	$|D \varphi_t (\omega,x)v|= e^{ \int_0^t (1 - |\varphi_s(\omega,x)|^2) \, ds}$.
	Hence,
	\begin{align*}
		\lambda_{top} &\geq \lim_{m \rightarrow \infty} \frac{1}{m} \log | D \varphi_m(\omega,x)v| 
		= \lim_{m \rightarrow \infty} \frac{1}{m} \int_0^m (1 - |\varphi_s(\omega,x)|^2) \, ds.
	\end{align*}
	\textit{Step 3:}
	Step 1 and 2 imply
	\begin{align*}
		\lambda_{top} 
		&\leq \liminf_{m \rightarrow \infty} \frac{1}{m} \int_0^m (1 - |\varphi_s(\omega,x)|^2) \, ds.
	\end{align*}	
	for some $\omega \in \Omega$ and $x \in M$ and equality in the case $n=1$. 
	Since $x \in M$ it holds that $\varphi_s(\omega,x) \in M$ for all $s \geq 0$.
	By the continuous-time ergodic theorem \cite[Section 2]{kornfeld1982} it follows that
	\begin{align*}
		\lambda_{top} \leq \frac{1}{Z_\sigma} \int_{\mathbb{R}^d} (1 - |x|^2) \, d \rho (x)
	\end{align*}
	and equality in the case $n=1$. 
\end{proof}
\begin{theorem}
	\label{le2}
	Let $\lambda_{top}$ be the top Lyapunov exponent of the RDS associated to (\ref{sde}). 
	\begin{enumerate}
		\item[(i)]
			For $n=1$ and $\sigma \leq \frac{1}{2}$, it holds that $\lambda_{top} > 0$.
		\item[(ii)]
			For $n=1$ and $\sigma \geq 2$, it holds that $\lambda_{top} < 0$.
		\item[(iii)]
		 	For $n \geq 2$, it holds that $\lambda_{top} < 0$.
	\end{enumerate}
\end{theorem}
\begin{proof}
	\textit{Case $n \geq 2$:} By Lemma \ref{le1} and changing to polar coordinates,
	\begin{align*}
		\lambda_{top} &\leq \frac{1}{Z_\sigma} \int_{\mathbb{R}^n} (1 - |x|^2) \, 
			e^{\frac{2}{\sigma^2}(\frac{1}{2} |x|^2 - \frac{1}{4} |x|^4) } \, dx \\
		& = \frac{1}{Z_\sigma} \, e^{\frac{1}{2 \sigma^2}} \int_{\mathbb{R}^n} (1 - |x|^2) \, 
			e^{- \frac{1}{2 \sigma^2} \left(|x|^2 -1\right)^2 } \, dx \\
		& = c \int_0^\infty (1 - r^2) r^{n-1} \, e^{- \frac{1}{2 \sigma^2} \left(r^2 -1\right)^2 } \, dr \\
		& = \tilde{c} \int_0^\infty r^{n-2} \left( \frac{d}{dr} \, 
			e^{- \frac{1}{2 \sigma^2} \left(r^2 -1\right)^2 } \right) \, dr
	\end{align*}
	with constants $c, \tilde{c} >0$. For $n=2$, 
	\begin{align*}
		\lambda_{top} \leq \tilde{c} \int_0^\infty \left( \frac{d}{dr} \, 
			e^{- \frac{1}{2 \sigma^2} \left(r^2 -1\right)^2 } \right) \, dr 
		= - \tilde{c} \, e^{-\frac{1}{2 \sigma^2}} <0.
	\end{align*}
	For $n \geq 3$, using integration by parts it follows that
	\begin{align*}
		\lambda_{top} &\leq \tilde{c} \left( \left[ r^{n-2} 
				e^{- \frac{1}{2 \sigma^2} \left(r^2 -1\right)^2 } \right]_0^\infty
			-  \int_0^\infty (n-2) r^{n-3} \, 
				e^{- \frac{1}{2 \sigma^2} \left(r^2 -1\right)^2 } \, dr \right) \\
			&= - \tilde{c} (n-2) \int_0^\infty r^{n-3} \, 
				e^{- \frac{1}{2 \sigma^2} \left(r^2 -1\right)^2 } \, dr <0.
	\end{align*}
	\textit{Case $n=1$:}
	By Lemma \ref{le1},
	\begin{align*}
		\lambda_{top} &= \frac{1}{Z_\sigma} \int_{-\infty}^{\infty} (1 - x^2) \, 
			e^{\frac{2}{\sigma^2}(\frac{1}{2} x^2 - \frac{1}{4} x^4) } \, dx \\
		&= \frac{2}{Z_\sigma} e^{\frac{1}{2 \sigma^2}} \int_0^\infty (1 - x^2) \, 
			e^{-\frac{1}{2\sigma^2}\left(x^2-1\right)^2 } \, dx.
	\end{align*}
	Using integration by parts, it follows that
	\begin{align}
		\int_1^\infty (1 - x^2) \, e^{-\frac{1}{2\sigma^2}\left(x^2-1\right)^2 } \, dx
		&= \frac{\sigma^2}{2} \int_1^\infty \frac{1}{x} \left( \frac{d}{dx} \, 
			e^{-\frac{1}{2\sigma^2}\left(x^2-1\right)^2 } \right) \, \nonumber dx \\ 
		&= \frac{\sigma^2}{2} \left( \left[ \frac{1}{x} 
			e^{-\frac{1}{2\sigma^2}\left(x^2-1\right)^2 } \right]_1^\infty  
			+\int_1^\infty \frac{1}{x^2} \, 
			e^{-\frac{1}{2\sigma^2}\left(x^2-1\right)^2 }  \, dx \right) \nonumber\\
		&= - \frac{\sigma^2}{2} +\frac{\sigma^2}{2} \int_1^\infty \frac{1}{x^2} \, 
			e^{-\frac{1}{2\sigma^2}\left(x^2-1\right)^2 }  \, dx.   \label{eq1}
	\end{align}
	We use integration by substitution to get lower estimates. Hence,
	\begin{align*}
		\int_0^1 (1 - x^2) \, e^{-\frac{1}{2\sigma^2}\left(x^2-1\right)^2 } \, dx
		&= \int_0^1 \frac{1}{2 \sqrt{1 - x}} \, x \, e^{-\frac{1}{2\sigma^2} x^2 } \, dx \\
		&= \int_0^\frac{3}{4} \frac{1}{2 \sqrt{1 - x}} \, x \, e^{-\frac{1}{2\sigma^2} x^2 } \, dx
			+ \int_\frac{3}{4}^1 \frac{1}{2 \sqrt{1 - x}} \, x \, e^{-\frac{1}{2\sigma^2} x^2 } \, dx \\
		&\geq \frac{1}{2} \int_0^\frac{3}{4} x \, e^{-\frac{1}{2\sigma^2} x^2 } \, dx
			+ \int_\frac{3}{4}^1 x \, e^{-\frac{1}{2\sigma^2} x^2 } \, dx \\ 
		&= \frac{\sigma^2}{2} + \frac{\sigma^2}{2} \left( e^{- \frac{1}{2 \sigma^2} \, \frac{9}{16}}
			- 2 \, e^{-\frac{1}{2 \sigma^2}} \right).
	\end{align*}
	Combining this estimate and (\ref{eq1}) yields to
	\begin{align*}
		\int_0^\infty (1 - x^2) \, e^{-\frac{1}{2\sigma^2}\left(x^2-1\right)^2 } \, dx 
		&\geq \frac{\sigma^2}{2} + \frac{\sigma^2}{2} \left( e^{- \frac{1}{2 \sigma^2} \, \frac{9}{16}}
			- 2 e^{-\frac{1}{2 \sigma^2}} \right) - \frac{\sigma^2}{2} \\
		&= \frac{\sigma^2}{2} \,  e^{- \frac{1}{2 \sigma^2}} 
			\left( e^{ \frac{1}{2 \sigma^2} \, \frac{7}{16}} - 2  \right) >0 
	\end{align*}
	for $\sigma \leq \frac{1}{2}$.
	Moreover, the following upper estimates on the integrals hold:
	\begin{align*}
		\int_0^1 (1 - x^2) \, e^{-\frac{1}{2\sigma^2}\left(x^2-1\right)^2 } \, dx 
		\leq \int_0^1 (1 - x^2) \, dx  
		= \left[x - \frac{1}{3} x^3 \right]_0^1 = \frac{2}{3}
	\end{align*}
	and
	\begin{align*}
		\int_1^\infty \frac{1}{x^2} \, e^{-\frac{1}{2\sigma^2}\left(x^2-1\right)^2 }  \, dx
		&= \int_1^{\sqrt{1+2 \sigma}}  \frac{1}{x^2} \, e^{-\frac{1}{2\sigma^2}\left(x^2-1\right)^2 }  \, dx 
			+\int_{\sqrt{1+2 \sigma}}^\infty  
				\frac{1}{x^2} \, e^{-\frac{1}{2\sigma^2}\left(x^2-1\right)^2 }  \, dx \\
		&\leq \int_1^{\sqrt{1+2 \sigma}}  \frac{1}{x^2} \, dx 
			+ e^{-2} \, \int_{\sqrt{1+2 \sigma}}^\infty  \frac{1}{x^2}  \, dx \\
		&= 1- \frac{1}{\sqrt{1 + 2 \sigma}} \left( 1 - e^{-2} \right).
	\end{align*}
	Combining these estimates and (\ref{eq1}), it follows that
	\begin{align*}
		\int_0^\infty (1 - x^2) \, e^{-\frac{1}{2\sigma^2}\left(x^2-1\right)^2 } \, dx 
		&\leq \frac{2}{3} - \frac{\sigma^2}{2} \frac{1}{\sqrt{1 + 2 \sigma}} \left( 1 - e^{-2} \right) \\
		&\leq \frac{2}{3} - \frac{\sqrt{2}}{\sqrt{3}} \left( 1 - e^{-2} \right) <0
	\end{align*}
	for $\sigma \geq 2$.
\end{proof}

\begin{remark}
	In the case $n=1$ there even exists some $\frac{1}{2}< \sigma^\ast <2$ such that $\lambda_{top} >0$ for
	$\sigma < \sigma^\ast$ and $\lambda_{top} <0$ for $\sigma > \sigma^\ast$. This can be seen by observing 
	that 
	\begin{align*}
		 \sigma \mapsto \int_{0}^{\infty} (1 - x^2) \, e^{-\frac{1}{2\sigma^2}(x^4-2x^2) } \, dx
	\end{align*}
	is strictly decreasing. However, this involves some more estimates of integrals.
\end{remark}

\begin{theorem}
	\label{asymp_stab}
	If the top Lyapunov exponent of the RDS $\varphi$ associated to (\ref{sde}) is negative,
	then there exists some $x \in M$  and $r>0$ such that $\varphi$ is asymptotically stable on $B(x,r)$. 
	In particular, this is the case for $n=1$ with $\sigma \geq 2$ and for $n \geq 2$.
\end{theorem}

\begin{proof}
	Remark \ref{invariant_measure} provides an invariant measure and Lemma \ref{assumptions} 
	shows that the assumptions of the stable manifold theorem
	(Theorem \ref{smt}) hold. In the considered cases, Theorem \ref{le2} yields that $\lambda_{top} < 0$.
	By stable manifold theorem (Theorem \ref{smt}), for every $\varepsilon \in (\lambda_{top},0)$
	there is a measurable map 
	$\beta: \Omega \times \mathbb{R}^d \rightarrow \mathbb{R}_+ \setminus \left\{ 0 \right\}$ 
	and $x \in M$ such that 
	\begin{align*}
		 \mathcal{S} (\omega,x) 
		 := \left\{ y \in \mathbb{R}^d : | \varphi_m(\omega,y) - \varphi_m(\omega,x)| \leq 
		 	\beta (\omega,x) e^{\varepsilon m} \textrm{ for all } m \in \mathbb{N} \right\}
	\end{align*}		 		
	is an open neighborhood of $x$ $\mathbb{P}$-a.s. Hence,
	\begin{align*}
		\mathbb{P} \left( \lim_{m \rightarrow \infty} 
		\textrm{diam} \left( \varphi_{t_m} ( \cdot, S (\cdot, x)) \right) =0 \right) >0.
	\end{align*}
	Lemma \ref{det_set} implies the existence of some $r>0$ such that $\varphi$ is 
	asymptotically stable on $B(x,r)$.
\end{proof}

\section{Synchronization}

We prove synchronization for the RDS associated to (\ref{sde}) in case of 
negative top Lyapunov exponent.
First, we show some similar properties to swift transitivity and contraction on large sets focusing 
on the set $M := \left\{ (x_1, x_2, \dots ,x_d) \in \mathbb{R}^d : x_i = 0 \textrm{ for } i>n \right\}$.
These will be used to show that the attractor is in any small ball centered at $M$
with positive probability. For negative top Lyapunov exponent, we use asymptotic stability in such a small
ball and apply \cite[Lemma 2.5]{SynNoise} to conclude synchronization.

\begin{lemma}
	\label{st}
	For all $x, y \in M$ and $r > 0$, there 
	is a time $t_0>0$ such that
	\begin{align*}
		\mathbb{P}(\varphi_{t_0} ( \cdot, B(x,r)) \subset B(y,2r)) > 0.
	\end{align*}
\end{lemma}

\begin{proof}
	Set $t_0 = \ln \frac{3}{2}$,
	\begin{align*}
		\psi(t) := x + \frac{t}{t_0} (y-x)
	\end{align*}
	for $t \in [0,t_0]$ and
	\begin{align*}
		\hat{\omega}^0(t) := \frac{1}{\sigma} \left( \psi(t) - x - \int_0^t b(\psi (s)) \, ds \right)
	\end{align*}
	for $t \in [0,t_0]$. Then, $\psi (t) \in M$ and $\hat{\omega}^0(t) \in M$ for all $t \in [0,t_0]$. 
	Set $\omega^0 $ to be the first $n$ components of $\hat{\omega}^0 $. 
	Then, $\varphi_t (\omega^0,x) = \psi (t)$ for all $t \in [0,t_0]$. In particular, 
	$ \varphi_{t_0}( \omega^0,x) = y$. By one-sided Lipschitz condition of $b$ (Lemma \ref{oslc}) 
	we have that
	\begin{align*}
		\frac{d}{dt} | \varphi_t(\omega, x^\prime) - \varphi_t(\omega,x)|^2 
		&= 2 \left\langle b(\varphi_t(\omega, x^\prime)) - b(\varphi_t(\omega, x)) ,
			\varphi_t(\omega, x^\prime) - \varphi_t(\omega, x) \right\rangle \\
		&\leq 2 |\varphi_t(\omega, x^\prime) - \varphi_t(\omega, x)|^2 
	\end{align*}
	for all $x^\prime \in B(x,r)$, $\omega \in \Omega$ and $t\geq 0$.
	By Gronwall's inequality, it follows that
	\begin{align*}
		| \varphi_t(\omega, x^\prime) - \varphi_t(\omega,x)| 
		\leq | x^\prime - x| \; e^t
		\leq r \; e^{t}
	\end{align*}
	for all $x^\prime \in B(x,r)$, $\omega \in \Omega$ and $t\geq 0$. 
	Then, for all $x^\prime \in B(x,r)$ and $\omega \in \Omega$
	\begin{align*}
		| \varphi_{t_0}(\omega, x^\prime) -y | 
		&\leq | \varphi_{t_0}(\omega, x^\prime) - \varphi_{t_0}(\omega, x)| 
			+ | \varphi_{t_0}(\omega, x) - \varphi_{t_0}(\omega^0, x)| \\
		&\leq \frac{3}{2} r + | \varphi_{t_0}(\omega, x) - \varphi_{t_0}(\omega^0, x)|.
	\end{align*}
	The map $\omega \mapsto \varphi_{t_0} (\omega,x)$ is continuous from 
	$C([0,t_0];\mathbb{R}^n) $ to $ \mathbb{R}^d$. 
	Then, there exists an $\delta > 0$ such that 
	\begin{align*}
		\mathbb{P}\left(\varphi_{t_0}(\cdot, B(x,r)) \subset B(y,2r)\right) 
		&\geq \mathbb{P} \left( | \varphi_{t_0}(\cdot, x) - \varphi_{t_0}(\omega^0, x)| 
			\leq \frac{r}{2} \right) \\
		&\geq \mathbb{P} \left( \sup_{s \in [0,t_0]} |\omega(s)-\omega^0(s)| \leq \delta \right) > 0.
	\end{align*}
	
\end{proof}

\begin{lemma}
	\label{cols}
	For every $R>0$ there is a ball $B(x,R)$ with 
	$x \in M$ and a time $t_0 > 0$ such that
	\begin{align*}
		\mathbb{P} \left( \textrm{diam} \left( \varphi_{t_0} \left( \cdot , B(x,R) \right) \right) \leq \frac{R}{4} \right) > 0.
	\end{align*}
	In particular, the RDS $\varphi$ is contracting on large sets.
\end{lemma}

\begin{proof}
	Let $R>0$ and $x := (2,0,0, \dots ,0)^T \in \mathbb{R}^d$. 
	Define
	\begin{align*}
		\hat{\omega}^0(t) := - \frac{t \, b(x)}{\sigma} 
	\end{align*}
	for $t \geq 0$. 
	Set $\omega^0 $ to be the first $n$ components of $\hat{\omega}^0 $.
	Then, $\varphi_t (\omega^0,x) = x$ for all $t \geq 0$.
	By Lemma \ref{oslc}, it holds that
	\begin{align*}
		\left\langle b(x) - b(y), x-y \right\rangle \leq -2 |x-y|^2
	\end{align*}
	for all $y \in \mathbb{R}^d$. This inequality and $\varphi_t (\omega^0,x) = x$ imply
	\begin{align*}
		\frac{d}{dt} |\varphi_t (\omega^0,x) - \varphi_t ( \omega^0,y ) |^2
		&= 2 \left\langle b( \varphi_t (\omega^0,x) ) - b(\varphi_t (\omega^0,y)), 
			\varphi_t (\omega^0,x) - \varphi_t (\omega^0,y) \right\rangle \\
		&\leq -4 \left|  \varphi_t (\omega^0,x) - \varphi_t (\omega^0,y) \right|^2.
	\end{align*}
	for $y \in B(x,R)$ and $t \geq 0$.
	Using Gronwall´s inequality it follows that
	\begin{align*}
		|x - \varphi_t ( \omega^0,y ) | 
		\leq |x-y| \, e^{-2t}
		\leq R \, e^{-2t}
	\end{align*}
	for all $ y \in B(x,R)$ and $t \geq 0$. Choose $t_0 \geq 0$ such that $  e^{-2t_0} \leq \frac{1}{16}$.
	Then, for all $y \in B(x,R)$ and $\omega \in \Omega$ 
	\begin{align*}
		\left| x - \varphi_{t_0} ( \omega,y ) \right|
		& \leq \left|x - \varphi_{t_0} ( \omega^0,y ) \right| 
			+ \left|\varphi_{t_0} ( \omega^0,y ) - \varphi_{t_0} ( \omega,y ) \right| \\
		& \leq  \frac{R}{16} + \left|\varphi_{t_0} ( \omega^0,y ) - \varphi_{t_0} ( \omega,y ) \right| .
	\end{align*}
	The map $\omega \mapsto \varphi_{t_0} (\omega,\cdot)$ is continuous 
	from $C([0,t_0];\mathbb{R}^n) $ to $ C(B(x,R); \mathbb{R}^d)$. 
	Then, there exists an $\delta > 0$ such that 
	\begin{align*}
		\mathbb{P}\left(\varphi_{t_0}(\cdot, B(x,R)) \subset B \left( x,\frac{R}{8} \right) \right) 
		&\geq \mathbb{P} \left( \sup_{y \in B(x,R)} 
			| \varphi_{t_0}(\omega^0, y) - \varphi_{t_0}(\omega, y)| \leq \frac{R}{16} \right) \\
		&\geq \mathbb{P} \left( \sup_{s \in [0,t_0]} |\omega(s)-\omega^0(s)| \leq \delta \right) > 0
	\end{align*}
	and thus
	\begin{align*}
		\mathbb{P} \left( \textrm{diam} \left(\varphi_{t_0}( \cdot , B(x,R))  \right) 
			\leq \frac{R}{4} \right) > 0.
	\end{align*}
\end{proof}

\begin{proposition}
	\label{smalldiameter}
	Let $A$ be the attractor of the RDS $\varphi$. Then, for each $\varepsilon > 0$ 
	there is an $x \in M$ such that
	\begin{align*}
		\mathbb{P} \left( A \subset B(x, \varepsilon ) \right) > 0.
	\end{align*}
\end{proposition}

\begin{proof}
	\textsl{Step 1:} In the first step it will be shown that  
	\begin{align*}
		\mathbb{P} \left( A \subset B(x_0, r_0) \right) > 0
	\end{align*}
	for some $r_0 > 0$, $x_0 \in M$ implies
	\begin{align*}
		\mathbb{P} \left( A \subset B\left( x_1, \frac{2}{3} r_0 \right) \right) >0
	\end{align*}
	for some $x_1 \in M$.
	\\
	Applying Lemma \ref{cols} with $R=2r_0$, there is an $y_1 \in M$, $t_1 >0$ such that
	\begin{align*}
		\mathbb{P} \left( \textrm{diam}\left(\varphi_{t_1} ( \cdot, B(y_1, 2r_0)) \right) 
			\leq \frac{r_0}{2} \right) > 0.
	\end{align*}
	Since $ \mathbb{P}$ is invariant under $\theta_{t_0}$ for every $t_0 > 0$, we have
	\begin{align*}
		\mathbb{P} \left( \textrm{diam}\left(\varphi_{t_1} ( \theta_{t_0} \cdot, B(y_1, 2r_0)) \right) 
			\leq \frac{r_0}{2} \right) > 0.
	\end{align*}
	Applying Lemma \ref{st}, there exists an $t_0 > 0$ such that
	\begin{align*}
		\mathbb{P} \left( \varphi_{t_0} ( \cdot, B(x_0,r_0) ) \subset B(y_1,2 r_0) \right) > 0.
	\end{align*}
	Moreover,
	\begin{align*}
		\left\{ \varphi_{t_0} ( \cdot, B(x_0,r_0)) \subset B(y_1,2r_0) \right\} \in \mathcal{F}_{0,t_0}
	\end{align*}
	and
	\begin{align*}
		\left\{ \textrm{diam}\left(\varphi_{t_1} ( \theta_{t_0} \cdot, B(y_1, 2r_0)) \right) 
			\leq \frac{r_0}{2} \right\} \in \mathcal{F}_{t_0,t_0+t_1}
	\end{align*}
	since $\left\{ \textrm{diam}\left(\varphi_{t_1} ( \cdot , B(y_1, 2r_0)) \right) 
	\leq \frac{r_0}{2} \right\} \in \mathcal{F}_{0,t_1}$
	and $\theta_{t_0}^{-1} \mathcal{F}_{0,t_1} = \mathcal{F}_{t_0,t_0+t_1}$. 
	Independence of $\mathcal{F}_{0,t_0}$ and $\mathcal{F}_{t_0,t_0+t_1}$ implies
	\begin{align*}
	\begin{split}
		&\mathbb{P} \left( \textrm{diam} \left( \varphi_{t_1 +t_0} ( \cdot, B(x_0,r_0)) \right) 
			\leq \frac{r_0}{2} \right)\\
		& \qquad =\mathbb{P} \left( \textrm{diam} \left( \varphi_{t_1} 
			( \theta_{t_0} \cdot, \varphi_{t_0}(\cdot, B(x_0,r_0))) \right) 
			\leq \frac{r_0}{2} \right) \\
		& \qquad \geq \mathbb{P} \left( \varphi_{t_0} ( \cdot, B(x_0,r_0) ) \subset B(y_1,2 r_0) \right)
			 \cdot
			\mathbb{P} \left(\textrm{diam}\left(\varphi_{t_1} ( \theta_{t_0} \cdot, B(y_1, 2r_0)) \right) 
			\leq \frac{r_0}{2} \right)
		 > 0.
	\end{split}
	\end{align*}
	Hence,
	\begin{align*}
		\mathbb{P} \left( \varphi_{t_1+t_0} (\cdot,B(x_0,r_0)) \subset 
			\bar{B}\left( \varphi_{t_1+t_0} (\cdot,x_0),\frac{r_0}{2}\right) \right) >0.
	\end{align*}
	By separability of $\mathbb{R}^n$, there exists a dense subset 
	$\left\{ z_m \right\}_{m \in \mathbb{N}} $ of $M$. 
	Since $ \varphi_{t_1+t_0} (\omega,x_0) \in M$ it follows that
	\begin{align*}
	\begin{split}
		&\left\{ \omega \in \Omega : \varphi_{t_1+t_0} (\cdot,B(x_0,r_0)) 
			\subset \bar{B}\left( \varphi_{t_1+t_0} (\cdot,x_0),\frac{r_0}{2}\right) \right\} \\
		& \quad\subset \left\{ \omega \in \Omega : \varphi_{t_1+t_0} (\cdot,B(x_0,r_0)) 
			\subset B \left( z_m, \frac{2}{3}r_0 \right) \textrm{ for some } m \in \mathbb{N} \right\}.	
	\end{split}
	\end{align*}
	By $\sigma$-additivity of $\mathbb{P}$ there exits an $x_1 \in M$ such that
	\begin{align*}
		\mathbb{P} \left( \varphi_{t_1+t_0} (\cdot,B(x_0,r_0)) 
			\subset B \left( x_1, \frac{2}{3}r_0 \right) \right) > 0.
	\end{align*}
	It holds that $\left\{ \varphi_{t_1+t_0} (\cdot,B(x_0,r_0)) \subset 
	B \left( x_1, \frac{2}{3}r_0 \right) \right\} \in \mathcal{F}_{0,t_1+t_0}$
	and $A$ is $\mathcal{F}_0$-measurable. By independence of $\mathcal{F}_0$ and 
	$\mathcal{F}_{0,t_1+t_0}$ and by the assumption of step 1, it follows that
	\begin{align*}
	\begin{split}
		&\mathbb{P} \left( \varphi_{t_1+t_0} (\cdot, A) \subset B\left(x_1,\frac{2}{3}r_0\right) \right)\\
		&\qquad \geq \mathbb{P} \left( A \subset B(x_0,r_0)\right)\cdot 
			\mathbb{P} \left( \varphi_{t_1+t_0} (\cdot, B(x_0,r_0)) 
			\subset B\left(x_1,\frac{2}{3}r_0\right) \right)
		>0.
	\end{split}
	\end{align*}
	$\varphi$-invariance of $A$ and $\theta_{t_1+t_0}$-invariance of $\mathbb{P}$ imply
	\begin{align*}
		\mathbb{P} \left( A \subset B\left(x_1,\frac{2}{3}r_0\right) \right)>0.
	\end{align*}
	\textsl{Step 2:} 
	Since the attractor $A$ is a random compact set, for each $\omega \in \Omega$ 
	the set $A(\omega)$ is bounded.
	Using $\sigma$-additivity of $\mathbb{P}$ it follows that there exists some $r_0 >0$ such that
	\begin{align*}
		\mathbb{P} \left( A \subset B(0, r_0) \right) > 0.
	\end{align*}
	Applying step one iteratively,
	\begin{align*}
		\mathbb{P} \left( A \subset B(x, \varepsilon) \right) > 0
	\end{align*}
	for some $x \in M$.
\end{proof}

\begin{corollary}
	\label{smalldiameter2}
	Let $A$ be the attractor of the RDS $\varphi$. For each 
	$x \in M$ and $\varepsilon > 0$,
	\begin{align*}
		\mathbb{P} \left( A \subset B(x, \varepsilon ) \right) > 0.
	\end{align*}
\end{corollary}

\begin{proof}
	By Proposition \ref{smalldiameter} there is an $x_0 \in M$
	such that $\mathbb{P} \left( A \subset B \left( x_0, \frac{\varepsilon}{2} \right) \right) > 0$. 
	By Lemma \ref{st} with starting ball 
	$B \left( x_0, \frac{\varepsilon}{2} \right)$ and arrival point $x$, there is a time $t>0$ such that
	\begin{align*}
		\mathbb{P} \left( \varphi_{t} \left( \cdot,B \left( x_0, \frac{\varepsilon}{2} \right) \right)
			\subset B \left( x, \varepsilon \right) \right) > 0.
	\end{align*}
	$\mathcal{F}_0$-measurability of $A$, $\mathcal{F}_{0,t}$-measurability of $\varphi_t$ and
	independence of $\mathcal{F}_0$ and $\mathcal{F}_{0,t}$ imply
	\begin{align*}
		\mathbb{P} \left( \varphi_t (\cdot,A) \subset B \left( x, \varepsilon \right) \right)
		\geq \mathbb{P} \left( A \subset B \left( x_0, \frac{\varepsilon}{2} \right) \right) \cdot 
			\mathbb{P} \left( \varphi_{t} \left( \cdot,B \left( x_0, \frac{\varepsilon}{2} \right)\right)
			\subset B \left( x, \varepsilon \right)\right) 
		>0.
	\end{align*}
	By $\varphi$-invariance of $A$ and $\theta_t$-invariance of $\mathbb{P}$ it follows that
	\begin{align*}
		\mathbb{P} \left( A \subset B \left( x, \varepsilon \right) \right)>0.
	\end{align*}
\end{proof}

\begin{lemma}
	\label{lemma_synchro}
	Let $A$ be the attractor of the RDS $\varphi$ and let $\varphi$ be asymptotically stable on $U$ with $\mathbb{P}(A \subset U) >0$.
	Then, $\varphi$ synchronizes.
\end{lemma}

\begin{proof}
	The attractor $A$ is an $\mathcal{F}_0$-measurable, $\varphi$-invariant, random closed set.
	By \cite[Lemma 2.5]{SynNoise} $A$ is a singleton.
\end{proof}

\begin{theorem}
	\label{synchronization}
	If the top Lyapunov exponent of the RDS $\varphi$ associated to (\ref{sde}) is negative, then
	$\varphi$ synchronizes.
	In particular, this is the case for $n=1$ with $\sigma \geq 2$ and for $n \geq 2$.
\end{theorem}
\begin{proof}
	By Theorem \ref{ex_attractor} there exists a pullback attractor $A$ of $\varphi$.
	In case of negative top Lyapunov exponent, Theorem \ref{asymp_stab} implies the existence of 
	some $x \in M$ and $r>0$ such that
	$\varphi$ is asymptotically stable on $B(x,r)$. By Corollary \ref{smalldiameter2} 
	\begin{align*}
		\mathbb{P} \left( A \subset B(x,r) \right) > 0.
	\end{align*}
	Applying Lemma \ref{lemma_synchro} it follows that synchronization occurs.
\end{proof}

\section{No synchronization}

We show that a positive top Lyapunov exponent implies lack of (weak) synchronization for
the RDS associated to (\ref{sde}). In order to prove this, we first need bounds
on the distance of two trajectories.

\begin{lemma}
	\label{less_than4}
	For $x,y \in \mathbb{R}^d$, $\omega \in \Omega$ and $t \geq 0.5$ it holds that
	\begin{align*}
		\left| \varphi_t ( \omega, x ) - \varphi_t ( \omega, y ) \right| \leq 4.
	\end{align*}
\end{lemma}

\begin{proof}
	\textit{Step 1:}
	Assume that $2^{k+2} \leq |x-y| \leq 2^{k+3}$ for some $k \geq 0$. 
	Define 
	\begin{align*}
		\tau_k (\omega) := \inf \left\{ t \geq 0 : 
		\left| \varphi_t ( \omega, x ) - \varphi_t ( \omega, y ) \right| \leq 2^{k+2} \right\}.
	\end{align*}
	Let $t \leq \tau_k (\omega)$. Then, $ \left| \varphi_t ( \omega, x ) \right| \geq 2^{k+1}$ or 
	$ \left| \varphi_t ( \omega, y ) \right| \geq 2^{k+1}$.
	Using Lemma \ref{oslc} it follows that
	\begin{align*}
	\begin{split}
		&\frac{d}{dt} \left| \varphi_t (\omega, x) - \varphi_t (\omega,y) \right|^2 \\
		& \qquad= 2 \left\langle b(\varphi_t (\omega, x)) - b(\varphi_t(\omega,y)),  
			\varphi_t (\omega, x) - \varphi_t (\omega,y) \right\rangle\\
		& \qquad \leq 2 \left( 1 - \frac{3}{4} \max \left\{  \left|\varphi_t (\omega, x)\right|^2, 
			 \left|\varphi_t (\omega, y) \right|^2 \right\} \right)
			\left| \varphi_t (\omega, x) - \varphi_t (\omega,y) \right|^2 \\
		& \qquad \leq 2 \left( 1 - \frac{3}{4} \left( 2^{k+1} \right)^2 \right)
			\left| \varphi_t (\omega, x) - \varphi_t (\omega,y) \right|^2.
	\end{split}
	\end{align*}
	By Gronwall's inequality
	\begin{align*}
		\left| \varphi_t (\omega, x) - \varphi_t (\omega,y) \right|
		\leq 2^{k+3} \; e^{\left( 1- 3 \cdot 4^k \right) t}.
	\end{align*}
	Then, for $t = \frac{\textrm{ln} \, 2}{ 3 \cdot 4^k -1 }$ it follows that 
	$\left| \varphi_t (\omega, x) - \varphi_t (\omega,y) \right| \leq 2^{k+2}$.
	Hence,
	\begin{align*}
		\tau_k (\omega) \leq \frac{\textrm{ln} \, 2}{\left( 3 \cdot 4^k -1 \right)}
			\leq \frac{\textrm{ln} \, 2}{2 \cdot 4^k}.
	\end{align*}
	\textit{Step 2:}
	Define
	\begin{align*}
		\tau (\omega) := \inf \left\{ t \geq 0 : 
		\left| \varphi_t ( \omega, x ) - \varphi_t ( \omega, y ) \right| \leq 4 \right\}.
	\end{align*}
	Using Step 1 iteratively it follows that
	\begin{align*}
		\tau (\omega) 
		\leq \sum_{k=0}^\infty  \frac{\textrm{ln} \, 2}{2 \cdot 4^k}
		= \frac{2}{3} \textrm{ln} \, 2 < \frac{1}{2}
	\end{align*}
	\textit{Step 3:}
	It remains to show that if 
	\begin{align*}
		\left| \varphi_r ( \omega, x ) - \varphi_r ( \omega, y ) \right| \leq 4
	\end{align*}
	for some $r \geq 0$ then
	\begin{align*}
		\left| \varphi_s ( \omega, x ) - \varphi_s ( \omega, y ) \right| \leq 4
	\end{align*}
	for all $s \geq r$.
	Assume there is a time $s > r$ such that 
	\begin{align*}
		\left| \varphi_s ( \omega, x ) - \varphi_s ( \omega, y ) \right| > 4.
	\end{align*}
	Define
	\begin{align*}
		\hat{\tau} (\omega) = \sup \left\{ t < s : 
			\left| \varphi_t ( \omega, x ) - \varphi_t ( \omega, y ) \right| \leq 4 \right\}.
	\end{align*}
	Then, $\left| \varphi_t ( \omega, x ) - \varphi_t ( \omega, y ) \right| \geq 4$ 
	for all $ t \in [\hat{\tau}(\omega), s]$. Hence, $\left| \varphi_t ( \omega, x ) \right| \geq 2$ 
	or $\left| \varphi_t ( \omega, x ) \right| \geq 2$ for any $ t \in [\hat{\tau}(\omega), s]$.
	Using Lemma \ref{oslc} it follows that
	\begin{align*}
	\begin{split}
		\frac{d}{dt} \left| \varphi_t (\omega, x) - \varphi_t (\omega,y) \right|^2 
		& = 2 \left\langle b(\varphi_t (\omega, x)) - b(\varphi_t(\omega,y)),  
			\varphi_t (\omega, x) - \varphi_t (\omega,y) \right\rangle\\
		&  \leq 2 \left( 1 - \frac{3}{4} \max \left\{  \left|\varphi_t (\omega, x)\right|^2, 
			 \left|\varphi_t (\omega, y) \right|^2 \right\} \right)
			\left| \varphi_t (\omega, x) - \varphi_t (\omega,y) \right|^2 \\
		&  \leq -4
			\left| \varphi_t (\omega, x) - \varphi_t (\omega,y) \right|^2.
	\end{split}
	\end{align*}
	for all $ t \in [\hat{\tau}(\omega), s]$. By Gronwall's inequality
	\begin{align*}
		\left| \varphi_s (\omega, x) - \varphi_s (\omega,y) \right|
		&\leq \left| \varphi_{\hat{\tau} (\omega)} (\omega, x) 
			- \varphi_{\hat{\tau} (\omega)} (\omega,y) \right| 				
			\; e^{-2 \left( s - \hat{\tau} (\omega) \right)} \\
		&= 4 \; e^{-2 \left( s - \hat{\tau} (\omega) \right)} \leq 4
	\end{align*}
	which is a contradiction to the definition of $s$.
\end{proof}

\begin{theorem}
	If the top Lyapunov exponent of the RDS $\varphi$ associated to (\ref{sde}) is positive, 
	then there is no synchronization of $\varphi$, not even weak synchronization.
	In particular, this is the case for $n=1$ with $\sigma \leq \frac{1}{2}$.
\end{theorem}

\begin{proof}
	Let $x := (0, 0 , \dots , 0,1)^T \in \mathbb{R}^d$ and denote by $(\cdot)^{(d)}$ the $d$-th component
	of a vector. Looking at the dynamics of $\varphi_t(\omega,x)$ one can observe that 
	$(\varphi_t(\omega,x))^{(d)} \leq 1$ for all $t \geq 0$ and $\omega \in \Omega$.
	By It\^{o}'s formula
	\begin{align*}
		\textrm{ln} \, (\varphi_t (\omega,x))^{(d)} 
		= \textrm{ln} \, (\varphi_0 (\omega,x))^{(d)} 
			+ \int_0^t \left( 1 - | \varphi_s(\omega,x) |^2 \right) \, ds
	\end{align*}
	for all $t \geq 0$ and almost all $\omega \in \Omega$. Hence,
	\begin{align*}
		\int_0^t \left( 1 - | \varphi_s(\omega,x) |^2 \right) \, ds \leq 0
	\end{align*}
	for all $t \geq 0$ and almost all $\omega \in \Omega$.
	Assume there is weak synchronization and denote by $a(\cdot)$ the weak point attractor 
	which is a singleton $\mathbb{P}$-almost surely.
	Since the RDS associated to the non-degenerate SDE does synchronize $a(\cdot)$
	is a single random point in $M$.
	Then,
	\begin{align*}
		\int_0^t \left( 1 - | a(\theta_s \omega) |^2 \right) \, ds 
			- \int_0^t \left( | \varphi_s(\omega,x) |^2 - | a(\theta_s \omega) |^2  \right) \, ds 
			\leq 0
	\end{align*}
	for all $t \geq 0$ and almost all $\omega \in \Omega$.
	By $\varphi$ invariance of $a(\cdot)$ and $a(\omega) \in M$, the distribution of $a(\cdot)$ 
	can be described by the invariant measure $\rho$ (see Remark \ref{invariant_measure}). 
	Using Fubini and the distribution of $a(\cdot)$ it follows that
	\begin{align*}
		\mathbb{E} \left[ \frac{1}{t} \int_0^t \left( 1 - | a(\theta_s \omega) |^2 \right) \, ds \right]
		&= \frac{1}{t} \int_0^t \mathbb{E} \left[ 1 - |a(\theta_s \omega) |^2  \right] \,ds \\
		&= \frac{1}{Z_\sigma} \int_{\mathbb{R}} (1 - y^2) \, 
			e^{\frac{2}{\sigma^2}(\frac{1}{2} y^2 - \frac{1}{4} y^4) } \, dy
	\end{align*}
	for all $t \geq 0$.
	By Lemma \ref{le1} and Theorem \ref{le2} this integral is equal to $\lambda_{top}$.
	Therefore,
	\begin{align}
		\label{greater_lambdatop}
		\mathbb{E} \left[ \frac{1}{t} 
			\int_0^t \left( | \varphi_s(\cdot,x) |^2 - | a(\theta_s \cdot) |^2  \right) \, ds \right] 
		\geq \lambda_{top} >0
	\end{align}
	for all $t \geq 0$.
	By weak synchronization $\varphi_s (\theta_{-s} \cdot, x)$ has to converge to $a(\cdot)$ as 
	$s \rightarrow \infty$ in probability. Using the continuous mapping theorem it follows that 
	$\left|\varphi_s (\theta_{-s} \cdot, x)\right|^2$ converges to $\left| a(\cdot) \right|^2$ 
	as $s \rightarrow \infty$ in probability. $\theta_s$ invariance of $\mathbb{P}$ implies that
	\begin{align*}
		\left|\varphi_s (\cdot, x)\right|^2 - \left| a(\theta_s \cdot) \right|^2
		\rightarrow 0 \quad \textrm{as } s \rightarrow \infty
	\end{align*}
	in probability. \\
	By Lemma \ref{less_than4} it follows that
	\begin{align*}
		\big| \left|\varphi_s (\cdot, x)\right|^2 - \left| a(\theta_s \cdot) \right|^2 \big|
		&= \big| \left|\varphi_s (\cdot, x)\right|- \left| a(\theta_s \cdot) \right| \big|
			\cdot \big| \left|\varphi_s (\cdot, x)\right|+ \left| a(\theta_s \cdot) \right| \big| \\
		&\leq \left| \varphi_s (\cdot, x)- a(\theta_s \cdot)  \right|
			\cdot \left( \left|\varphi_s (\cdot, x) - a(\theta_s \cdot) \right| 
			+ 2 \left| a(\theta_s \cdot) \right| \right) \\
		&\leq 16 + 8 \left|a (\theta_s \cdot )\right| 
	\end{align*}
	for $s \geq \textrm{ln} \, 0.5$. Then,
	\begin{align*}
	\begin{split}
		&\mathbb{E} \left[ \big| \left|\varphi_s (\cdot, x)\right|^2 
			- \left| a(\theta_s \cdot) \right|^2 \big| \, 
			\mathbbm{1}_{\big| \left|\varphi_s (\cdot, x)\right|^2 
			- \left| a(\theta_s \cdot) \right|^2 \big| \geq K} \right] \\
		&\qquad \leq \mathbb{E} \left[ \left( 16 + 8 \left| a(\theta_s \cdot) \right| \right) \,
			\mathbbm{1}_{\left| a(\theta_s \cdot) \right| \geq \frac{K -16}{8}} \right] \\
		&\qquad = \frac{1}{Z_\sigma} \int_{\mathbb{R}} (16 + 8 |y|) \, 
			\mathbbm{1}_{|y| \geq \frac{K -16}{8}} \;
			e^{\frac{2}{\sigma^2}(\frac{1}{2} y^2 - \frac{1}{4} y^4) } \, dy
	\end{split}
	\end{align*}
	for $s \geq \textrm{ln} \, 0.5$.
	By rapidly decaying property of $e^{\frac{2}{\sigma^2}(\frac{1}{2} y^2 - \frac{1}{4} y^4) }$
	this integral converges to $0$ as $K \rightarrow \infty$. Hence, 
	$\left( \left|\varphi_s (\cdot, x)\right|^2 - \left| a(\theta_s \cdot) \right|^2 \right)
	_{s\geq \textrm{ln} \, 0.5}$ is uniformly integrable. Therefore,
	$\left|\varphi_s (\cdot, x)\right|^2 - \left| a(\theta_s \cdot) \right|^2$ converges to $0$
	as $s \rightarrow \infty$ in $L^1$. By $L^1$ convergence there exists some $t_0 \geq 0$ such that
	\begin{align*}
		\mathbb{E} \left[ | \varphi_s(\cdot,x) |^2 - | a(\theta_s \cdot) |^2 \right]
		\leq \frac{\lambda_{top}}{2}
	\end{align*}
	for all $s \geq t_0$. Using Fubini it follows that
	\begin{align*}
	\begin{split}
		&\mathbb{E} \left[ \frac{1}{t} 
			\int_0^t \left( | \varphi_s(\cdot,x) |^2 - | a(\theta_s \cdot) |^2  \right) \, ds \right] \\
		& \qquad \leq \frac{1}{t} \, \mathbb{E} \left[ 
			\int_0^{t_0} \left( | \varphi_s(\cdot,x) |^2 - | a(\theta_s \cdot) |^2  \right) \, ds \right]
			+ \frac{t-t_0}{t} \, \frac{\lambda_{top}}{2}
	\end{split}
	\end{align*}
	for $t > t_0$. For large $t$ this term will get smaller than $\lambda_{top}$ 
	which is a contradiction to (\ref{greater_lambdatop}).
\end{proof}

\section{Summary and open problems}

We considered the stochastic differential equation (\ref{sde}) with drift given 
by the multidimensional double-well potential with degenerate additive noise.
Similar to the case with non-degenerate noise, 
if the noise is acting in more than one direction, $n \geq 2$, then synchronization occurs. A more interesting
phenomenon appears in the case where noise affects the SDE in one direction, $n =1$. 
In this case, we proved synchronization for $\sigma \geq 2$ and showed that there will be
no synchronization, not even weak synchronization, for $\sigma \leq \frac{1}{2}$.
\\
In the case $n=1$, there actually exists a critical value $\frac{1}{2} < \sigma^\ast < 2$ where the
behavior changes. That means that there will be synchronization for $\sigma > \sigma^\ast$ 
and there will be no synchronization, not even weak synchronization, for $\sigma < \sigma^\ast$.
It is not clear what happens for $\sigma = \sigma^\ast$ where the top Lyapunov exponent is zero.
\\
Numerical simulations of the $2$-dimensional case suggest that there is
weak synchronization for small noise on $\mathbb{R} \times \mathbb{R}_+$.
It remains an open problem to describe the attractor in this case.

\nocite{*}
\bibliographystyle{plain}
\bibliography{mybib}

\end{document}